\documentclass{amsart}
\usepackage{amssymb}
\usepackage{enumerate}
\usepackage{upref}
\usepackage[T1]{fontenc}
\usepackage[utf8]{inputenc}
\usepackage{mathtools}
\usepackage{tikz}
\usetikzlibrary{matrix}
\usepackage{tikz-cd}

\numberwithin{equation}{section}

\newtheorem{theorem}{Theorem}
\newtheorem{lemma}[theorem]{Lemma}
\newtheorem{proposition}[theorem]{Proposition}

\theoremstyle{definition}
\newtheorem{definition}[theorem]{Definition}

\newtheorem{question}[theorem]{Question}

\theoremstyle{remark}

\newcommand{\I}{[0,1]}

\newcommand{\R}{\mathbb{R}}
\newcommand{\C}{\mathbb{C}}
\newcommand{\N}{\mathbb{N}}

\newcommand{\Olo}{\mathcal{O}}
\newcommand{\hol}{\mathcal{O}}

\newcommand{\Aut}{\operatorname{Aut}}
\newcommand{\VF}{\operatorname{VF}}
\newcommand{\CVF}{\operatorname{CVF}}
\newcommand{\SL}{\operatorname{SL}}

\newcommand{\GL}{\operatorname{GL}}

\begin{document}

\title[Parametric Jet Interpolation for manifolds with DP]{Parametric Jet Interpolation for Stein manifolds with the Density Property}
\author[A. Ramos-Peon]{Alexandre Ramos-Peon}
\address{Matematisk Institutt, Universitetet i Oslo. Postboks 1053, Blindern. 0316 OSLO, Norway}
\email{alexaram@math.uio.no}
\author{Riccardo Ugolini}
\address{IMFM, University of Ljubljana, Jandranska Ulica 19, 1000 Ljubljana, Slovenia}
\email{riccardo.ugolini@imfm.si}


\date{\today}


\begin{abstract}
Given a Stein manifold with the density property, we show that under a suitable topological condition it is possible to prescribe derivatives at a finite number of points to automorphisms depending holomorphically on a Stein parameter. This is an Oka property of the manifold and is related to its holomorphic flexibility. 
\end{abstract}

\maketitle

\section{Introduction}\label{sec:intro}

Given a (connected) complex manifold $X$, it is often of interest to study its group of holomorphic automorphisms, denoted by $\Aut(X)$. In many cases this object can not be described in a simple way, but one may determine properties of this group and of its action on $X$. A basic property is transitivity, which can be seen as a special case of $N$-transitivity.
\begin{definition}
Given $N \in \N$ and a group $G$ equipped with a left action on a set $X$, we say that the action is $N$-transitive if for any two subsets $\{a_1,\dots,a_N\}, \{b_1,\dots,b_N\}$ of $X$ consisting of $N$ distinct elements, there exists $g\in G$ such that $ga_j=b_j$ for $j=1,\dots,N$. The action is \textit{infinitely transitive} if it is $N$-transitive for every $N \in \N$.
\end{definition}
When one deals with complex manifolds and their group of automorphisms, it is possible to consider not only pointwise interpolation, but also \textit{jet interpolation}. This means that we wish to find holomorphic automorphisms of $X$ with prescribed values of all derivatives (up to a given order). The first result in this direction, due to Forstneri\v c \cite{F-Interpolation} for $X=\C^n$, uses the dense subgroup of $\Aut(\C^n)$ generated by shears, which are some well-known automorphisms that appeared in the seminal paper by Rosay and Rudin about self-maps of $\C^n$ \cite{RosayRudin}.

An example of more general complex manifolds for which $\Aut(X)$ acts infinitely transitively is provided by Stein manifolds with the \textit{density property}. A complex manifold $X$ has the \textit{density property} if the Lie algebra generated by {complete} vector fields (those whose flows are defined for all times) is dense in the Lie algebra of all vector fields in the compact-open topology.

This notion (see Definition \ref{def:DP} below for a more detailed discussion), first introduced by Varolin \cite{Varolin1}, has turned out to be fruitful, allowing for new insights into classical questions. Many constructions are possible in Stein manifolds with the density property and there has been an ongoing effort to determine which manifolds have the density property. For a complete account, we refer the interested reader to the monograph \cite[Chapter 4]{F} and the references therein.

In another paper \cite{VarolinII}, Varolin proved a jet interpolation theorem for holomorphic automorphisms of a Stein manifold with the density property. In this paper, we generalize his result to holomorphic families of jets.

Given an $N$-tuple $\hat{x}=(\hat{x}_1,\dots,\hat{x}_N)$ of distinct points in $X$,  
we can consider the space $Y$ of all possible collections of nondegenerate $k$-jets at the points $\{ \hat{x}_i \}_{i=1}^N$ such that their respective images are distinct (see Section 2 for  formal definitions). The following is the main theorem of this paper. 
\begin{theorem} \label{main}
Let $W$ be a Stein manifold, $X$ a Stein manifold with the density property, $k,N\in \N_{\geq 0}$, and $(\hat{x}_1,\dots,\hat{x}_N)$ an $N$-tuple of distinct points in $X$. For each $i=1,\dots,N$ 
let $\gamma_i^w$ be a nondegenerate $k$-jet at $\hat{x}_i$ depending holomorphically on the parameter  $w\in W$. Then there exists a null-homotopic parametric family of automorphisms $F^w\in\Aut(X)$ depending holomorphically on $w \in W$, such that the $k$-jet of $F^w$ at $\hat{x}_i$ equals $\gamma_i^w$ for all $i=1,\dots,N$ and $w\in W$ if and only if $\gamma=(\gamma_1,\dots,\gamma_N):W\to Y$ is null-homotopic.
\end{theorem}

The special case $k=0$ corresponds to pointwise interpolation and was proved by Kutzschebauch and the first author \cite{K-R}. For arbitrary $k,N \in \N$ and $X=\C^n, \ n>1$, this result was proved by the second author \cite{Ugolini}.

The present paper is part of a common effort to describe so-called Oka properties of groups of automorphisms. The main results in this direction are due to Forstneri\v c and L\'arusson in  \cite{FrancFinnur}, where the authors focus on $\Aut(\C^n)$ and many of its subgroups. We refer to \cite[Chapters 5 \& 6]{F} for a comprehensive survey on Oka theory.

We now outline the proof strategy for Theorem \ref{main}:
\begin{enumerate} [(i)]
\item Modify the homotopy of jets $\gamma^t:W\to Y$, assumed to exist by hypothesis, to an isotopy of holomorphic maps $\gamma^t:W\to Y$ (smooth on $t\in [0,1]$), connecting the given jet $\gamma^1:W\to Y$ to the constant map $\gamma^0:W\to Y$ consisting of the jet of the identity map on $X$; 
\item Use the homotopy $\gamma^t$ from (i) to construct a homotopy of injective holomorphic maps $F_t$, defined on an open neighborhood of the fixed $N$-tuple $\hat{x}=(\hat{x}_1,\dots, \hat{x}_N)$, such that the jet of $F_t$ is equal to $\gamma^t$;
\item Approximate the homotopy from (ii) with a family of parametrized automorphisms on a large compact set $L_1 \times K_1 \subset W \times X$;
\item repeat the above steps inductively on larger and larger compact sets $L_j \times K_j \subset W \times X$, $j \in \N$.
\end{enumerate}
This results in countably many parametrized automorphisms and we must ensure that their composition converges on $W \times X$. A necessary condition for convergence is that the holomorphic maps from (ii) approach the identity as the compacts $L_j \times K_j \subset W \times X$ become larger (i.e. at each induction step). Correspondingly, we must ensure that at each induction step (with the exception of the initial one), for each $t$, the jets from the isotopy in (i) be close to the jet of the identity map on  $L_j \times K_j $. We point out that we just need to assume for the parametrized family of jets to be null-homotopic and not homotopic to the jet of the identity as in (i), this is equivalent as $Y$ is path-connected when $X$ is.

This is the strategy followed in \cite{K-R} for the pointwise interpolation case (i.e. $k=0$). When $k\geq 1$, a new significant difficulty arises in this process, namely in step (ii). We explain this in detail in the remainder of this section. The rest of the paper is organized as follows: in Section 2, we lay out the notation and define the relevant space of jets and prove its ellipticity. In Section 3, we handle step (ii), and reduce the rest of the problem to a technical construction whose proof is the object of Section 4. 

Suppose that the given homotopy $\gamma^t:W\to Y$ is such that for each fixed $t$, the jet $\gamma^t:W \to Y$ is close to the jet of the identity on a compact set $L\subset W$.
Given $K \subset X$, we wish to find a homotopy of injective holomorphic maps which are close to the identity on $L\times K$.
For simplicity assume that $N=1$ and that $\gamma^t(w)$ fixes the point $\hat{x}$ for all $(t,w) \in [0,1] \times W$.
For fixed values of $t$ and $w$, the jet $\gamma^t(w)$ can be represented by a polynomial (in $z \in X$) in a local chart.
If we restrict ourselves to a small enough compact $L \subset W$, this polynomial can be chosen to be holomorphic in $w \in L$ and smooth in $t \in [0,1]$.
According to the proof of \cite[Theorem 1]{Ugolini}, on this local chart $U\subset X$ there exist locally defined vector fields $\{V_1,\dots, V_M\}\subset \VF (U)$ and holomorphic functions $\{f_1,\dots, f_M\}\subset \Olo (L)$ such that the jet of $\phi_{V_1}^{tf_1(w)} \circ \dots \circ \phi_{V_M}^{tf_M(w)}$ at $\hat{x}$ is $\gamma^{t,w}$, where $\phi_V^s$ denotes the flow of the vector field $V$ at time $s\in \C$.
Choosing $U$ to be Runge and using the density property one can replace the locally defined vector fields with globally defined complete ones $\{W_1,\dots, W_M\}\subset \CVF (X)$ in such a way that the jet of $\phi_{W_1}^{tf_1(w)} \circ \dots \circ \phi_{W_M}^{tf_M(w)}$ at $\hat{x}$ approximates $\gamma^{t,w}$ for $w \in L$. If we further require this composition to be close to the identity on a large compact $K \subset X$,
it will be sufficient that the functions $f_i$, $i=1,\dots,M$ are sufficiently close to zero for each $w \in L$. Unfortunately, the above construction fails to produce functions with this property already at the level of $1$-jets for the following reason.

We can identify a parametrized $1$-jet fixing a point with a map $G:W \to \GL_n(\C)$. It is easy to reduce to the case $G:W \to \SL_n(\C)$. In the proof of \cite[Theorem 1]{Ugolini} the author uses the following solution to the Vaserstein problem, proved by Ivarsson and Kutzschebauch in a spectacular application of Oka theory:

\begin{theorem}[Ivarsson and Kutzschebauch \cite{Ivarsson}] \label{vaser}
Let $W$ be a finite dimensional reduced Stein space and $G:W \rightarrow SL_n(\C)$ be a
nullhomotopic holomorphic mapping. Then there exist an integer $M \in \N$ and holomorphic mappings 
\[
	G_1,\dots, G_M: W \rightarrow \C^{n(n-1)/2}
\]
such that $G$ can be written as the finite product of upper and lower diagonal unipotent matrices with entries in $\hol(W)$:
\[
	G(w)=
\left( \begin{array}{ccc}
	1 & 0 \\
	G_1(w) & 1 \end{array} \right)
\left( \begin{array}{ccc}
	1 & G_2(w) \\
	0 & 1 \end{array} \right) 
	\dots
\left( \begin{array}{ccc}
	1 & G_M(w) \\
	0 & 1 \end{array} \right).
\]
\end{theorem}
Consider the following statement: if  $G(w_0)=Id$ holds for some $w_0 \in W$, then $G_1(w_0)=\dots=G_M(w_0)=0$. In the above setting, if the jet $\gamma^t$ happens to be equal to the identity for some $w_0$, this statement would imply that the corresponding functions $f_i$ from the described construction should evaluate to zero at $w_0$, and hence the approximation would be close to the identity on 
$K\subset X$ for $w \in W$ close to $w_0$, as desired. Unfortunately, the naive statement just considered is false and we now provide a counterexample.

Suppose $W$ is an open disc in $\C$ of radius $1$ and center $1\in\C$. Since $W$ is contractible, any map $W\to SL_n(\C)$ is nullhomotopic. For $n=2$, let
\[
G(w)=\left( \begin{array}{ccc}
	\frac{1}{w} & 0 \\
	0 & w \end{array} \right)
\]
and suppose it can be written as a product of upper and lower diagonal unipotent matrix functions which are the identity when $w=1$. By a simple induction, this implies that the term in position $(2,2)$ is always of the form $1+(w-1)^2 f(w)$ for an $f\in \Olo(W)$ whose Laurent polynomial around $1$ does not include summands with negative exponent. As this term also needs to be equal to $w$ for all $w \in W$, we reach a contradiction.

This fact is deeper than it appears and a more detailed account of this phenomenon can be found in the proof of Theorem \ref{vaser}. Here we illustrate it for the case $n=2$. For a fixed $M \in \N$, consider the map

\begin{align*}
\psi_M:\C^M &\to SL_2(\C) \\
\psi_M(z_1,\dots,z_M)&=\left( \begin{array}{ccc}
	1 & 0 \\
	z_1& 1 \end{array} \right)
\left( \begin{array}{ccc}
	1 & z_2 \\
	0 & 1 \end{array} \right) 
	\dots
\left( \begin{array}{ccc}
	1 & z_M \\
	0 & 1 \end{array} \right).
   \end{align*}
Theorem \ref{vaser} implies the existence of a holomorphic lift in the following diagram:

\[
\begin{tikzcd}[row sep=1cm, column sep=1.5cm]
 & \C^M \arrow{d}{\psi_M}   \\
 W \arrow{r}{G} \arrow{ru} & SL_2(\C)
\end{tikzcd}
\]
The existence of a continuous lift was proven by Vaserstein \cite{Vaserstein}. To find a holomorphic lift, the authors use the Oka-Grauert-Gromov principle for sections of holomorphic submersions coming from the diagram

\[
\begin{tikzcd}[row sep=1cm, column sep=1.5cm]
 & \C^M \arrow{d}{\pi \circ \psi_M}   \\
 W \arrow{r}{\pi \circ G} \arrow{ru} & \C^2 \setminus \{0\}
\end{tikzcd}
\]
where $\pi: SL_2(\C) \to \C^2 \setminus \{0\}$ is given by projection to the last row. This choice was made to simplify the discussion of the fibers of the submersion, nonetheless the map $\pi \circ \psi_M$ is a submersion only outside the set $$S_M=\{(z_1,\dots,z_M) \in \C^M: z_1=\dots=z_{M-1}=0\}.$$ 

The main consequence of this fact is that Theorem \ref{vaser} ensures the existence of functions $G_i:W \to \C,\ i=1,\dots,M$ which are never all zero for the same value $w_0 \in W$. Therefore, the previously discussed approach for step (ii) is not viable, and a more elaborate procedure is required. 

\section{Notation and set up}

Let $X$ be a complex manifold of dimension $n$ and fix $k \in \N$. We now give precise definitions of the objects that were discussed in the previous section.
\begin{definition}
Let $F,G:U \subset X \to X$ be representatives for holomorphic germs at $p \in U$. We say that $F$ and $G$ have the same $k$-jet at $p$ if their Taylor expansion in some local chart about $p$ agrees up to order $k$. This defines an equivalence class which we call a $k$-jet and denote by $[F]_p$. We say that the jet is nondegenerate if its linear part (the Jacobian of any representative) has nonzero determinant.
\end{definition}
Thus in a local chart, a $k$-jet can be uniquely represented by a polynomial mapping of total degree (i.e. the maximal degree of its $n$ components) at most $k$. 

The set of all nondegenerate $k$-jets at a point $p \in X$ will be denoted by $J_{p,\ast} (X)$. Observe that this is a complex manifold. Since a jet $\gamma \in J_{p,\ast} (X)$ is an equivalence class of germs, $\gamma(p) \in X$ is well-defined and we call this the \textit{image} of $\gamma$ at $p$. Furthermore, the map
\begin{align*}
  \pi:J_{p,\ast} (X) &\to X \\
 \gamma &\mapsto \gamma(p)
\end{align*}
is surjective and holomorphic.
For $q \in X$ we can think of $\pi^{-1}(q) =: J_{p,q} (X)$ as the set of non-degenerate $k$-jets at $p$ whose image is $q$. For convenience and unless noted otherwise we will use the word \textit{jet} for \textit{nondegenerate $k$-jet}. 

As we are interested in jet interpolation by automorphisms at more than one point, let us define the relevant spaces.
Fix $N \in \N$ distinct points $\{ \hat{x}_i \}_{i=1}^N \subset X$. We interpret this $N$-tuple as a point $\hat{x} \in X^N \setminus \Delta$, where 
\[
\Delta= \bigcup_{1\leq i<j \leq N} \left\{(z_1,\dots,z_N) \in X^N : z_i=z_j \right\}. 
\]
If we apply $\pi$ coordinate by coordinate we obtain a map (which we still denote by $\pi$) from $J_{\hat{x}_1,\ast} (X) \times \dots \times J_{\hat{x}_N,\ast} (X)$ to $X^N$.
Now let
\[
Y:=J_{\hat{x}_1,\ast} (X) \times \dots \times J_{\hat{x}_N,\ast} (X) \setminus \pi^{-1} (\Delta).
\]
be the complex manifold representing all possible collections of nondegenerate $k$-jets at the points $\{ \hat{x}_i \}_{i=1}^N$ such that their respective images are distinct.
If $U$ is an open set containing all of the points $\{ \hat{x}_i \}_{i=1}^N$ and $F:U \to X$ is an injective holomorphic map, we denote by $[F]_{\hat{x}}$ the $N$-tuple of jets $([F]_{\hat{x}_1}, \dots, [F]_{\hat{x}_N}) \in Y$.

Given a jet of the form $\gamma = [F]_{\hat{x}} \in Y$, we denote by $(\gamma)^{-1}$ the jet $[F^{-1}]_{F(\hat{x})}$ and observe that  $(\gamma)^{-1} \in Y$ if $\gamma(\hat{x})=\hat{x}$.

If some metric compatible with the manifold structure on $X$ is given, the space $J_{p,\ast}(X)$ inherits a metric, and it induces the natural distance on $X^N\setminus\Delta$ defined by taking the maximum distance of each coordinate projection. Therefore to $Y$ is associated a distance function $d$. It follows from the Cauchy estimates that uniform convergence on compacts of $X$ implies convergence in $Y$ with respect to this metric.

Let $W$ be a complex manifold and $\gamma:W \to Y$ a holomorphic map. We are interested in finding a \textit{holomorphic} map $F:W \to \Aut(X)$ such that $[F^w]_{\hat{x}} = \gamma^w$ for all $w \in W$.
As $\Aut(X)$ is not a complex manifold, we define $F:W \to \Aut(X)$ to be holomorphic if $F^w(x)$ is holomorphic as a map from $W \times X$ to $X$. We denote the space of all such maps with $\Aut_{W}(X)$ and observe that it can be seen as a subgroup of $\Aut(W \times X)$.

It is clear that a necessary condition for the existence of such a map is that $\Aut(X)$ is \textit{large}: more precisely we will require $X$ to be a Stein manifold with the density property, which we now define. Let $\VF(X)$ be the Lie algebra of holomorphic vector fields on $X$ (holomorphic sections of $T^{1,0}X$). In this paper \textit{vector field} means \textit{holomorphic vector field} in this sense. Recall that a vector field on a manifold is called \textit{complete} if at every point $x$ the solution $\phi_t(x)$ to the flow equation starting at $x$ is defined for all $t\in \C$. 

\begin{definition}\label{def:DP} Let $X$ be a complex manifold and $\mathfrak{g} \subset \VF(X)$ be a Lie subalgebra of vector fields on $X$. We say that $\mathfrak{g}$ has the density property if the subalgebra of $\mathfrak{g}$ generated by complete vector fields is dense in $\mathfrak{g}$ with respect to the uniform topology on compact sets. We say that $X$ has the density property if $ \VF(X)$ does.
\end{definition}

The density property by itself does not imply that $\Aut(X)$ is large, as one can deduce by considering a compact manifold. However under the additional assumption that $X$ is Stein and of dimension at least $2$, the density property implies that $\Aut(X)$ is infinite dimensional (see e.g. \cite{Varolin1}). An important feature of Stein manifolds with the density property is that locally defined holomorphic maps $U\to X$ can be approximated uniformly on compacts by automorphisms: this is known as the Andersén-Lempert theorem. The following is the parametric version proven in \cite[Theorem 2.2]{K-R}, which we will use in its full generality. 
\begin{proposition}
\label{AL}
Let $W$ be a Stein manifold and $X$ a Stein manifold with the density property. 
Let $U\subset W\times X$ be an open set and $F_t:U\to W\times X$ be a  smooth  homotopy of injective holomorphic maps acting as the identity on the $W$ coordinate and with $F_0$ being the inclusion map.
Suppose $K\subset U$ is a compact set such that $F_t(K)$ is $\hol(W\times X)$-convex for each $t\in \I$.  Then there exists a neighborhood $V$ of $K$ such that for all $t\in\I$,
$F_t$ can be approximated uniformly on compacts of $V$ (with respect to any distance function on $X$) by automorphisms $\alpha_t\in\Aut_{W}(X)$ which depend smoothly on $t$, and moreover we can choose $\alpha_0=id$.
\end{proposition}
Let us return to our setting. Given $\gamma=(\gamma_1, \dots, \gamma_N) \in Y$, any $V \in \VF(X)$ defines a flow $\phi_V^t$ in a neighborhood of $\{ \pi(\gamma_i) \}_{i=1}^N$ for small enough values of $t$. Hence the jet $[\phi_V^t \circ \gamma]_{\hat{x}} \in Y$ is well defined (for small $t$).
Differentiating with respect to $t$, we obtain $\tilde{V} \in \VF(Y)$ such that $\phi_{\tilde{V}}^t (\gamma)= [\phi_V^t \circ \gamma]_{\hat{x}}$ for all $\gamma\in Y$: we call $\tilde{V}$ \textit{the lift} of $V$. 
We observe that if $V$ is complete, then so is $\tilde{V}$. We denote by $\CVF (X)$ the set of complete vector fields on $X$.
The set $\widetilde{\VF(X)} = \{ \tilde{V} \in \VF(Y) : V \in \VF(X) \}$ is a Lie subalgebra of $\VF(Y)$ and we note that if $X$ has the density property, so does $\widetilde{\VF(X)}$. Similarly, an automorphism $\alpha$ of $X$ lifts to an automorphism $\tilde\alpha$ of $Y$. 

We now prove that complete vector fields on $X$ can be lifted to span the tangent space of $Y$.

\begin{lemma}\label{lem:elliptic}
Let $X$ be a Stein manifold with the density property and let $\gamma \in Y$.
Then there exist $M= \dim Y \in \N$ complete vector fields $\{ \theta_i \}_{i=1}^M \subset \CVF(X)$ such that $\{\tilde{\theta}_i(\gamma)\}_{i=1}^M$ is a basis for the tangent space of $Y$ at the point $\gamma$.
In particular the map 
 \begin{align*}
 \C^M &\to Y \\
 (t_1,\dots,t_M) &\mapsto \phi_{\tilde{\theta}_1}^{t_1} \circ \dots \circ \phi_{\tilde{\theta}_M}^{t_M}(\gamma)
\end{align*}
is a biholomorphism from a neighborhood of $0$ to a neighborhood of $\gamma$.
\end{lemma}

\begin{proof}
It is sufficient to prove the lemma for $\gamma = [Id]_{\hat{x}}$.
Indeed, Varolin proved in \cite{VarolinII} that there exists $F \in \Aut(X)$ such that $[F\circ \gamma]_{\hat{x}}=[Id]_{\hat{x}}$
.
If the vector fields $\{\tilde{\theta}_i\}_{i=1}^M$ span the tangent space to $Y$ at $[Id]_{\hat{x}}$, then $\{\widetilde{F^*\theta}_i\}_{i=1}^M$ span $T_\gamma Y$.

We first claim that the conclusion is true for $N=1$ and $X=\C^n, \ n>1$.
To see this, suppose $\hat{x}=(0,\dots,0)$ and consider the vector fields $V_{I,j}=z^I \frac{\partial}{\partial z_j}$ where $I$ runs through the multi-indexes of degree less or equal than $k$ and $j=1,\dots, n$.
Not all of them are complete, but they do span the tangent space of $Y$ at $[Id]_{\hat{x}}$.
As $X$ has the density property, we can approximate each $V_{I,j}$ by a sum of complete vector fields.
For a good enough approximation, this new collection must also span the tangent space of $Y$ at $[Id]_{\hat{x}}$. 
The claim is proved by picking out a basis from this generating set.

A similar proof works for arbitrary $X$ (Stein with the density property): let $U$ be a Runge coordinate neighborhood of $\hat{x}$ such that $\hat{x}$ corresponds to $0 \in \C^n$ under the chart map. On $U$ we consider the pullback of the vector fields $V_{I,j}$ above. As $U$ is Runge and $X$ has the density property, we conclude as above by approximating these pulled-back fields  by sums of complete vector fields.

Let now $X$ be as above and $N>1$. We choose coordinate neighborhoods $U_i$ of $\hat{x}_i$, $i=1,\dots N$ such that $U = U_1 \cup \dots \cup U_N$ is Runge in $X$ and each $\hat{x}_j$ is mapped to zero under the appropriate coordinate chart.
For each $i=1, \dots, N$, we pull back $V_{I,j}$ on $U_i$ and extend it to the zero field on the other coordinate neighborhoods in order to obtain a generating set for $Y$ at $[Id]_{\hat{x}}$, and proceed as above to obtain complete vector fields. 
\end{proof}
The collection $\{ \tilde{\theta_i} \}_{i=1}^M\subset \CVF(Y)$ spans $T_\gamma Y$ for all $\gamma$ outside of an analytic set $A\subset Y$. By a procedure involving a step-by-step lowering of the (finite) dimension of $A$ (see the second part of the proof of \cite[Lemma 3.2]{K-R}, or \cite[Thm. 4]{KK-Zeit}), 
implies that the collection  $\{ \theta_i \}\subset \CVF(X)_{i=1}^M$ can be enlarged to a finite collection $\{ \theta_i \}\subset \CVF(X)$ such that $\{\tilde{\theta}_i\}$ spans the tangent space of $Y$ at \textit{every point} in $Y$. In particular, this proves that $Y$ is an \textit{elliptic manifold} in the sense of Gromov (though not necessarily Stein) and hence an \textit{Oka-Forstneri\v{c} manifold}. These notions are more general than we require here, but for our purposes, it will suffice to record that an Oka-Forstneri\v c manifold $Y$ satisfies the following h-principle:
given a homotopy $f_t:W\to Y\ (t\in\I)$, where $W$ is Stein and $f_0,f_1$  are holomorphic, there exists a new homotopy $g_t:W\to Y$ which is smooth on $t$, holomorphic on $W$ for all $t\in \I$ and $f_0=g_0,f_1=g_1$. We call such a special homotopy a \textit{smooth isotopy of holomorphic maps} (or jets, in the case that  $Y$ is as defined previously in this section). 
To avoid confusion, we refrain from talking about smooth isotopies of maps $W\to \Aut(X)$ depending (smoothly) on a variable $t\in [0,1]$, which we prefer to call families of parametrized automorphisms.

\section{First local approximation}
We begin by showing (Lemma \ref{lem:stepzero} below) that we can approximately achieve the conclusion of Theorem \ref{main} when the parameter lies in a compact set. We will make use of the following result from \cite{K-R}:
\begin{theorem}
\label{thm:KR}
Let $W$ be a Stein manifold and $X$ a connected Stein manifold with the density property. Let $N$ be a natural number and 
$x:W\to X^N \setminus \Delta$ be a holomorphic map, and fix $\hat{x}=(\hat{x}_1,\dots,\hat{x}_N)$. 
Then there exists a holomorphic map $\alpha:W\to \Aut(X)$, homotopic to the identity, such that $\alpha^w(\hat{x}_i)=x_i^w$ for all $i=1,\dots, N$ and all $w\in W$,
if and only if $x$ is nullhomotopic. 
\end{theorem}
In our notation, a holomorphic map $x:W\to X^N\setminus \Delta$ is nothing but a parametrized $0$-jet at the point $\hat{x}=(\hat{x}_1,\dots,\hat{x}_N)$.

\begin{lemma} \label{lem:stepzero}
Let $W$ be Stein and $X$ Stein with the density property. Fix $\hat{x}$ so that the space $Y$ of $k$-jets at $\hat{x}$ is defined as in Section 2. Let $\gamma^1: W\to Y$ be holomorphic and null-homotopic, with $\gamma^t$ denoting the homotopy from $\gamma^1$ to the constant jet $\gamma^0=[Id]_{\hat{x}}$. Let $\varepsilon>0$ and a holomorphically convex compact set $L=\hat{L}\subset W$ be given. 
Then there exists a family of parametrized automorphisms $F:[0,1] \times W \to \Aut(X)$ with $F^0=id$ such that
\[
d([F^{t,w} \circ \gamma^{t,w}]_{\hat{x}}, [Id]_{\hat{x}})< \varepsilon
\]

for all $(t,w) \in [0,1] \times L$, where $d$ denotes the distance in $Y$ defined in Section 2.
\end{lemma}
\begin{proof}
Since both  $\gamma^0$ and $\gamma^1$ are holomorphic, $W$ is Stein and $Y$ is elliptic by Lemma \ref{lem:elliptic}, the h-principle applies: we can therefore assume that $\gamma^t$ is in fact a smooth isotopy of holomorphic maps.

Define $x_t$ to be $\pi(\gamma^t)$, i.e. the image of the jets $\gamma^t$ at $\hat{x}$.  As $x_t$ takes values in $X^N \setminus \Delta$, we may apply Theorem \ref{thm:KR}. Therefore we may assume without any loss of generality that for all $(t,w) \in [0,1] \times W$ the image of $\gamma^{t,w}$ (hence the one of $(\gamma^{t,w})^{-1}$)  is the fixed $N$-tuple $\hat{x}$.

This allows us to uniquely represent the jets $(\gamma^{t,w})^{-1}$ by parametrized polynomial mappings of total degree at most $k$ fixing $0\in \C^n$. Indeed, 
let $U\subset X$ be a disjoint union of coordinate neighborhoods $U_j$ of the fixed points $\hat{x}_j$ and $\phi_j:U_j\to \phi_j(U_j)\subset \C^{n}$ be the charts sending $\hat{x}_j$ to $0$.
For each $j=1,\dots,N$ and $(t,w) \in [0,1] \times W$ there exists a uniquely determined polynomial mapping $Q_j^{t,w}$ of total degree at most $k$ such that
\[
\left[ \phi_j\circ(\gamma^{t,w})^{-1}\circ\phi_j^{-1}\right]_0=\left[ Q_j^{t,w} \right]_0.
\] 
By uniqueness, these polynomials mappings which fix $0\in \C^n$ depend smoothly on $t \in [0,1]$ and holomorphically on $w \in W$.
For fixed $w\in W$ and for each $j=1,\dots,N$, since nondegenerate mappings are locally invertible, there exists a neighborhood ${V}_j$ of $0$ in $\C^{n}$ such that for all $t\in [0,1]$, $ Q_j^{t,w}|_{{V}_j}$ is injective and 
$ Q_j^{t,w}({V}_j)\subset \phi_j(U_j)$.
Given a compact set $L' \supset L$ such that $\mathring{L'} \supset L$, by compactness
we can shrink $V_j$ such that the above holds for all $(t,w)\in\I\times L'$.
Let $V$ be the disjoint union of $\phi_j^{-1}({V_j})$, and define the \textit{injective holomorphic map} $P^t$ on $\mathring{L'} \times V$ by setting
\[
P^t(w,z):=(w,\phi_j^{-1}\circ  Q_j^{t,w} \circ \phi_j(z)):\mathring{L'} \times V \to W \times X.
\]
Notice that the union of the ``graph of the $N$ fixed points'' $$K=\bigcup_{j=1}^N\{(w,\hat{x}_j): w \in L\}\subset W \times X$$ is a $\hol(W \times X)$-convex set which is fixed by $P^t$. 
Since $P^0$ is the identity, we can apply Proposition \ref{AL} and obtain a family of parametrized automorphisms $F^{t,w}$ which approximates $(\gamma^{t,w})^{-1}$ uniformly on compacts in a neighborhood of $K$. By the Cauchy estimates, this implies the approximation of the jet with respect to the distance function $d$ (see Section 2).
\end{proof}

\section{Proof}

The following technical proposition is similar to \cite[Proposition 4.4]{K-R} but with jet approximation instead of just pointwise approximation.
\begin{proposition}
\label{prop:jetapprox}
Let $L_1, L_2 \subset W$ be $\Olo(W)$-convex compact sets such that $L_1 \subset \mathring{L}_2$ and let $K,C \subset X$ be $\Olo(X)$-convex compact sets with $K \subset \mathring{C}$.
Let $\gamma:[0,1] \times W \to Y$ be a smooth isotopy of holomorphic jets such that $\gamma^{0,w} = [Id]_{\hat{x}}$ for all $w \in W$.
Then for every $\varepsilon, \alpha>0$ there exists $\delta=\delta(K,L_1,\varepsilon)>0$ such that if
\begin{equation}\label{smallness}
d(\gamma^{t,w},[Id]_{\hat{x}})<\delta \quad  \forall (t,w) \in [0,1] \times L_1, 
\end{equation}
then
\begin{enumerate}
\item[(i)] there exists $\psi:[0,1] \times L_1 \to \Aut(X)$ such that 
\begin{align} 
d(\psi^{t,w}(z),z) <\varepsilon \text{ for }& (t,w,z) \in [0,1]\times L_1 \times K \\
[\psi^{t,w}]_{\hat{x}} = \gamma^{t,w} \text{ for }& (t,w) \in [0,1]\times L_1
\end{align}
\item[(ii)] there exists $F:[0,1] \times W \to \Aut(X)$ with $F^{0,w}=Id$ for all $w \in W$ and such that 
\begin{align} 
d([F^{1-t,w}\circ \gamma^{1,w}]_{\hat{x}}, \gamma^{t,w}) <\alpha \text{ for }& (t,w) \in [0,1]\times L_2\\
d(F^{1-t,w}(z),\psi^{t,w}(z)) <\varepsilon \text{ for }& (t,w,z) \in [0,1]\times L_1 \times C
\end{align}
 such that $\psi_0=F_0=Id$.
\end{enumerate}
\end{proposition}

\begin{proof}
Let us examine the nature of the local biholomorphism near $[Id]_{\hat{x}} \in Y$ given by Lemma \ref{lem:elliptic}.
In particular observe that given a compact $K \subset X$, for $(t_1,\dots,t_M) \in U \subset \C^M$ small enough, the automorphism $\phi_{\theta_1}^{t_1} \circ \dots \circ \phi_{\theta_M}^{t_M} \in \Aut(X)$ is going to be arbitrarily close to the identity on $K$.
We can then pick $\delta$ such that 
\[
\{\gamma^{t,w} :(t,w) \in [0,1] \times L_1\} \subset \{\phi_{\tilde{\theta}_1}^{t_1} \circ \dots \circ \phi_{\tilde{\theta}_M}^{t_M}([Id]_{\hat{x}}):(t_1,\dots,t_M) \in U \} \subset Y.
\]
Apply Lemma \ref{lem:elliptic} to obtain a family of parametrized automorphisms $\psi:[0,1]\times L_1 \to \Aut(X)$ such that $\psi_0=Id$, $dist(\psi^{t,w}(z),z)<\varepsilon/2$ and $[\psi^{t,w}]_{\hat{x}}=\gamma^{t,w}$ for $(t,w,z) \in [0,1]\times L_1 \times K$. This proves (i).

Consider the non-autonomous parametrized vector field on $X$
\[
  \Theta^{t,w}(z)=\frac{d}{ds}\bigg|_{s=t} \psi^{1-s,w}((\psi^{1-t,w})^{-1}(z))
\]
defined for $(t,w,z)\in [0,1]\times L_1 \times X$ and observe that the lift of $\Theta$ to $Y$ satisfies
\[
\tilde{\Theta}^{t,w}(\gamma^{1-t,w})=\frac{d}{ds}\bigg|_{s=t} \gamma^{1-s,w}.
\]
Therefore $\Theta$ is well defined on $[0,1]\times L_1 \times X \cup [0,1]\times W \times \{ \hat{x}_1,\dots, \hat{x}_N \}$.

By \cite[Lemma 2.2]{FR}, as in the proof of Lemma \ref{lem:stepzero}, there is a Runge neighborhood $\Omega \subset W \times X$ of  $L_1 \times K \cup L_2 \times \{ \hat{x}_1,\dots, \hat{x}_N \}$ such that the flow $f^t:\Omega \to W \times X$ of $\Theta^t$ is injective and $f^t(\Omega)$ is Runge for every $t \in [0,1]$.
Using Proposition \ref{AL}, we obtain the desired $F^t:W \to \Aut(X)$. Observe that condition (4.4) only depends on this last step, hence we can approximate arbitrarily well and choose $\alpha$ only when invoking (ii).
\end{proof}

To prove the main theorem, we will construct families of automorphisms $F:[0,1] \times W \to \Aut(X)$ using the above result inductively on a growing sequence of compacts of $W \times X$. 
In order to apply Proposition \ref{prop:jetapprox} again (on a larger compact, in views of exhausting the parameter space $W$), we need a smooth isotopy of parametrized jets that are close to the identity for all $t \in [0,1]$ over the compact $L_2$. This homotopy does not come for free during the induction step for the following reason.

Let $\gamma^t$ be as above starting at the constant jet $\gamma^0=[Id]_{\hat{x}}$ and apply Proposition \ref{prop:jetapprox} to obtain a family of parametrized automorphisms $F^t$. We now have a new homotopy of jets
\[
h:[0,1] \times W \to Y
\]
given by
\[
h^{t,w}=
\begin{cases*}
\gamma^{2t,w} & for $0 \leq t \leq \frac{1}{2}$ \\
[F^{2t-1,w} \circ \gamma^{1,w}]_{\hat{x}} & for $\frac{1}{2} \leq t \leq 1$
\end{cases*}
\]
connecting $h^{0,w}=[Id]_{\hat{x}}$ to $h^{1,w} \approx[Id]_{\hat{x}}$ for $w \in L_2$.
We cannot immediately use Proposition \ref{prop:jetapprox} for
$h^{t,w}$ over the larger $L_2$, as the smallness condition (\ref{smallness}) required for all $t$ is only satisfied at the end point $t=1$.

However, this issue can be handled as follows:

\begin{lemma} [Lemma 4.2, \cite{K-R}] \label{lem:homotopy}
Let $L \subset W$ be a $\Olo(W)$-convex compact set and $h^t:W \to Y$ be a smooth homotopy between the constant $h^0=[Id]_{\hat{x}}$ and a holomorphic map $h^1$. Then there exists $\varepsilon=\varepsilon(h,L)>0$ such that for every $0<\alpha < \varepsilon$ and every smooth $F:[0,1] \times W \to Y$ with $F^t = h^{2t}$ for $t \leq \frac{1}{2}$ satisfying
\[
d(F^{t,w},F^{1-t,w}) < \alpha/2 \text{ for } (t,w) \in [0,1]  \times L
\]
and every  $\Olo(W)$-convex compact set $L' \subset \mathring{L}$, there exists an analytic homotopy $H: [0,1] \times W \to Y$ between $[Id]_{\hat{x}}$ and $h^1$ such that
\[
d(H^{t,w}, [Id]_{\hat{x}}) < \alpha \text{ for } (t,w) \in [0,1]  \times L'
\]
\end{lemma}

Note that \cite[Lemma 4.2]{K-R} is \textit{stated} for the manifold $Y$ which stands for $X^N\setminus \Delta$. However, the proof uses only general topological constructions, as well as the Oka property which holds for any elliptic manifold $Y$.  

We now proceed with the proof of Theorem \ref{main}.

\begin{proof}[Proof of Theorem \ref{main}]
Let $L_j \subset W, K_j \subset X, \ j \in \N$ be exhaustions by compact holomorphically convex compact sets such that $L_j \subset \mathring{L}_{j+1}$ and $K_j \subset \mathring{K}_{j+1}$. Assume that $L_0= \emptyset$ and $(\hat{x}_1, \dots, \hat{x}_N) \in \mathring{K}_0$. Fix a sequence $\varepsilon_j>0$ such that $\varepsilon_j<d(K_{j-1},X \setminus K_j)$ and $\sum \varepsilon_j < + \infty$.

Since $\gamma : W \to Y$ is null-homotopic, there exists an isotopy on holomorphic maps $\gamma^t:W \to Y$ such that $\gamma^0=[Id]_{\hat{x}}$ and $\gamma^1=\gamma$. We can immediately apply Lemma \ref{lem:stepzero} to obtain $\varphi_0:[0,1] \times W \to \Aut(X)$ such that
\[
d([\varphi_0^{t,w} \circ \gamma^{t,w}]_{\hat{x}}, [Id]_{\hat{x}})<min(\varepsilon_0,\delta(K_0,L_1,\varepsilon_1/2)) \text{ for } (t,w) \in [0,1] \times L_1
\]
where $\delta$ is as in Proposition \ref{prop:jetapprox}.

Conclusion (i) in the latter gives $\psi:[0,1] \times L_1 \to \Aut(X)$ such that 
\begin{align*} 
d(\psi^{t,w}(z),z) <\varepsilon_1/2 \text{ for }& (t,w,z) \in [0,1]\times L_1 \times K_0 \\
[\psi^{t,w}]_{\hat{x}} =[\varphi_0^{t,w} \circ \gamma^{t,w}]_{\hat{x}} \text{ for }& (t,w) \in [0,1]\times L_1
\end{align*}
Let $C_1 \subset X$ be a $\Olo(X)$-convex compact containing the $\varepsilon_1/2$-envelope of $$K_1 \cup \psi^{[0,1], L_1}(K_1).$$
Let $\alpha_1:= min(\varepsilon_1, \delta(C_1, L_2,\varepsilon_2/2),\varepsilon([\varphi_0^{t,w} \circ \gamma^{t,w}]_{\hat{x}},L_2))/2$, where $\varepsilon([\varphi_0^{t,w} \circ \gamma^{t,w}]_{\hat{x}},L_2)$ is the one arising from Lemma \ref{lem:homotopy}.

By (ii) in Proposition \ref{prop:jetapprox} there exists $\varphi_1:[0,1] \times W \to \Aut(X)$ such that
\[
d([\varphi_1^{1-t,w}\circ \varphi_0^{1,w} \circ \gamma^{1,w}]_{\hat{x}},[\varphi_0^{t,w} \circ \gamma^{t,w}]_{\hat{x}} ) < \alpha_1
\]
for $(t,w) \in [0,1]\times L_2$ and
\[
d(\varphi_1^{1-t,w}(z), \psi^{t,w}(z) ) <\varepsilon_1/2 \text{ for } (t,w,z) \in [0,1]\times L_1 \times C_1.
\]
We just proved the initial step of the following induction.

Suppose for $j=1, \dots, k$ we have $C_j \subset X$ holomorphically convex sets and $\varphi_j:[0,1] \times W \to \Aut(X)$ such that
\begin{enumerate}[(a)]
\item $\varphi_k^{0,w}=Id \in \Aut(X)$ for all $w \in W$;
\item $d([\varphi_k^{1-t,w} \circ \varphi_{k-1}^{1,w} \circ \dots \circ \varphi_0^{1,w} \circ \gamma^{1,w}]_{\hat{x}},h^{t,w})<\alpha_k$ for $(t,w) \in [0,1] \times L_{k+1}$;
\item $\mathring{C}_k \supset K_{k} \cup F_k^{t,w}(  (F_{j-1}^{t,w})^{-1} (K_k))$ for each $j=1, \dots, k$, $t \in [0,1]$ and $w \in L_j \setminus L_{j-1}$;
\item for each $j=1, \dots, k$, if $w \in L_j \setminus L_{j-1}$ then  $d(\varphi_k^{t,w}(z),z) < \varepsilon_k$ for each $t \in [0,1]$ and $z  \in K_k \cup  F_{k-1}^{t,w}(  (F_{j-1}^{t,w})^{-1} (K_k))$.
\end{enumerate}
where $h:[0,1] \times W \to Y$ is a homotopy between $[Id]_{\hat{x}}$ and $[\varphi_{k-1}^{1,w} \circ \dots \circ \varphi_0^{1,w} \circ \gamma^{1,w}]_{\hat{x}}$, $F_k^{t,w}(z)=\varphi_k^{t,w} \circ \varphi_{k-1}^{t,w} \circ \dots \circ \varphi_0^{t,w}(z)$, $$\alpha_k=min(\varepsilon_k,\delta(C_k,L_{k+1}, \varepsilon_{k+1}/2), \varepsilon(h,L_{k+1}))/2.$$

We first explain how the induction provides a familiy of parametrized automorphisms $G:[0,1] \times W \to \Aut(X)$ such that $G^{0,w}= Id$ and $G^{1,w}$ satisfies the thesis of the theorem. Thanks to (c) and (d), \cite[Lemma 4.1]{K-R} ensures that the sequence $F_k:[0,1] \times W \to \Aut(X)$ converges to $F:[0,1] \times W \to \Aut(X)$ uniformly on compacts, while condition (b) evaluated at $t=0$ shows that the inverse of $F^{1,w}$ provides the required parametrized automorphism. The fact that such inverse is null-homotopic is guaranteed by (a).

Let us now assume that we have the required objects for $j=1, \dots, k$, we begin by considering the homotopy $\tilde{H}:[0,1] \times W \to \Aut(X)$ given by
\[
\tilde{H}^{t,w}=
\begin{cases*}
h^{2t,w} & for $0 \leq t \leq \frac{1}{2}$ \\
[\varphi_k^{2t-1,w} \circ  \varphi_{k-1}^{1,w} \circ \dots \circ \varphi_0^{1,w} \circ \gamma^{1,w}]_{\hat{x}} & for $\frac{1}{2} \leq t \leq 1$
\end{cases*}
\]
Observe that condition (b) and the definition of $\alpha_k$ ensure we can apply Lemma \ref{lem:homotopy}, hence there exists an homotopy $H:[0,1] \times W \to \Aut(X)$ such that $H^{0,w}=Id$, $H^{1,w}=[\varphi_k^{1,w} \circ  \varphi_{k-1}^{1,w} \circ \dots \circ \varphi_0^{1,w} \circ \gamma^{1,w}]_{\hat{x}}$ and
\[
d(H^{t,w}, [Id]_{\hat{x}}) < \delta(C_k,L_{k+1},\varepsilon_{k+1}/2) \text{ for } (t,w) \in [0,1]  \times L_k
\]
Part (i) of Proposition \ref{prop:jetapprox} provides the existence of $\psi:[0,1] \times L_k \to \Aut(X)$ such that 
\begin{align} 
d(\psi^{t,w}(z),z) <\varepsilon_{k+1}/2 \text{ for }& (t,w,z) \in [0,1]\times L_k \times C_k \\
[\psi^{t,w}]_{\hat{x}} = H^{t,w} \text{ for }& (t,w) \in [0,1]\times L_k
\end{align}
Let $C_{k+1} \subset X$ be a holomorphically convex compact set containing the $\varepsilon_{k+1}/2$-envelope of
\begin{equation}
\tag{*}
C_k \cup  K_{k+1} \cup \psi^{[0,1],L_k}(K_{k+1}) \cup \psi^{[0,1],L_{j-1}}(K_{k+1})
\end{equation}
and of
\begin{equation}
\tag{*}
\psi^{t,w}(F_k^{t,w}((F_{j-1}^{t,w})^{-1}(K_{k+1}))
\end{equation}
for each $j=1, \dots, k$ and $(t,w) \in [0,1] \times L_{j-1}$.

Part (ii) of Proposition \ref{prop:jetapprox} gives $\varphi_{k+1}:[0,1] \times W \to \Aut(X)$ with $\varphi_{k+1}^{0,w}=Id$ for all $w \in W$ and such that 
\begin{align} 
d([\varphi_{k+1}^{1-t,w}\circ H^{1,w}]_{\hat{x}}, H^{t,w}) <\alpha_{k+1} \text{ for }& (t,w) \in [0,1]\times L_{k+2}\\
d(\varphi_{k+1}^{1-t,w}(z),\psi^{t,w}(z)) <\varepsilon_{k+1}/2 \text{ for }& (t,w,z) \in [0,1]\times L_k \times C_{k+1}
\end{align}
with $\alpha_{k+1}= min(\varepsilon_{k+1},\delta(C_{k+1},L_{k+2}, \varepsilon_{k+2}/2), \varepsilon(H,L_{k+2}))/2$.

We now explain the reason these choices provide the $k+1$-th step.
Condition (a) is clearly satisfied.
Equation (4.8) and the fact that $H^{1,w}=[\varphi_{k}^{1,w} \circ \dots \circ \varphi_0^{1,w} \circ \gamma^{1,w}]_{\hat{x}}$ provide condition (b).
Equation (4.9) tells that the image of $\varphi_{k+1}$ is $\varepsilon_{k+1}$-close to the one of $\psi$, hence $C_{k+1}$ also contains the sets in (*) if we substitute $\psi$ with $\varphi_{k+1}$, thus condition (c) is fulfilled.
Condition (d) is obtained as a combination of (4.6) and (4.9).

\end{proof}

\section{Discussion and Questions}

Theorem \ref{main} covers the case of a \textit{finite} number of points; in a previous paper by the second author \cite{Ugolini} as well as in the first paper on jet interpolation in $\C^n$ \cite{F-Interpolation} the (parametric) jet interpolation is provided for a special class of sequences.

\begin{definition}[\cite{RosayRudin}] \label{def:tame}
A closed discrete sequence of points $(a_j)_{j\in\N} \subset \C^n$ without repetition is {\em tame} if there exists a holomorphic automorphism 
$F \in \Aut(\C^n)$ such that 
\[	F(a_j)=(j,0,\dots,0) \ \ \text{for all}\ j=1,2,\ldots.
\]
\end{definition}

The reason to restrict to such sequences is given by the fact that not all sequences in $\C^n$ are equivalent under the action of $\Aut(\C^n)$ and the one consisting of natural numbers in the first axis presents good properties of flexibility. In particular $\Aut(\C^n)$ acts transitively on $\N \times \{ 0\}^{n-1}$, thus it is common to speak about \textit{tame sets}.

In \cite{Ugolini} the author provides parametric jet interpolation at such sequences, yet he does not consider the possibility of having a countable amount of parametrized points. In fact, it is not well understood whether tame sequence are \textit{generic} in some suitable sense. To our knowledge and opinion, the result shedding the most light on the issue is due to Winkelmann:

\begin{theorem}[\cite{WinkelmannTame}]

Let $\{a_k\}_{k \in \N} \subset \C^n, \ n>1$ be a discrete sequence without repetition such that

\begin{equation}
\sum_k \frac{1}{| | a_k| | ^{2n-2}} < \infty
\end{equation}

Then $\{a_k\}_{k \in \N}$ is tame.
\end{theorem}

The reason this result provides an insight on the topology of the class of tame sequences is that condition (5.1) is open in $\ell^2$.

\begin{question} \label{question}
Let $W$ be a Stein manifold and $a_k: W \to \C^n$ for $k \in \N$ a sequence of holomorphic maps such that $a_k(w) \neq a_j(w)$ for $j \neq k$ and all $w \in W$. Furthermore assume that
\[
\sum_k \frac{1}{| | a_k(w)| | ^{2n-2}} < \infty
\]
for all $w \in W$. What topological conditions do we need on $\{a_k\}_{k \in \N}$ to ensure the existence of a holomorphic $F: W \to \Aut(\C^n)$ such that $F^w (k,0,\dots,0)=a_k(w)$ for all $w \in W$ and $k \in \N$?
\end{question}

A similar question can be posed for the newly introduced notion of weakly and strongly tame sequences in a complex Stein manifold with the density property:

\begin{definition}[\cite{WTame}]
Let $X$ be a complex manifold.
An infinite discrete
subset $D$ is called weakly tame
if for every exhaustion function $\rho : X \to \R$
and every function $\zeta : D \to \R$
there exists an automorphism $\Phi$ of $X$
such that $\rho(\Phi(x)) \geq \zeta(x)$ for all $x \in D$.
\end{definition}

\begin{definition}[\cite{STame}]
\label{def-tame}
Let $X$ be a complex manifold. We call a closed discrete infinite set $D \subset X$ a strongly tame set if for every injective mapping $f \colon D \to D$ there exists a holomorphic automorphism $F \in \Aut(X)$ such that $F_{|_D} = f$.
\end{definition}

It is worth mentioning that both notions are equivalent to the standard definition of tameness for $X=\C^n$.

If $X$ is a Stein manifold with the density property, any two strongly tame sequences are $\Aut(X)$ equivalent \cite[Proposition 2.4]{STame}; providing an answer to Question \ref{question} for wealky tame sequences should then be harder as the previous result is not known for this class of manifolds.

\bibliography{bib2}

\providecommand{\bysame}{\leavevmode\hbox to3em{\hrulefill}\thinspace}
\providecommand{\MR}{\relax\ifhmode\unskip\space\fi MR }
\providecommand{\MRhref}[2]{%
  \href{http://www.ams.org/mathscinet-getitem?mr=#1}{#2}
}
\providecommand{\href}[2]{#2}
\begin{thebibliography}{{Win}17}

\bibitem[AU18]{STame}
R.~B. {Andrist} and R.~{Ugolini}, \emph{{A new notion of Tameness}}, ArXiv
  e-prints (2018).

\bibitem[FcL14]{FrancFinnur}
Franc Forstneri\v~c and Finnur L\'arusson, \emph{Oka properties of groups of
  holomorphic and algebraic automorphisms of complex affine space}, Math. Res.
  Lett. \textbf{21} (2014), no.~5, 1047--1067. \MR{3294562}

\bibitem[For99]{F-Interpolation}
Franc Forstneri{\v{c}}, \emph{Interpolation by holomorphic automorphisms and
  embeddings in {${\bf C}^n$}}, J. Geom. Anal. \textbf{9} (1999), no.~1,
  93--117. \MR{1760722}

\bibitem[For11]{F}
Franc Forstneri\v{c}, \emph{Stein manifolds and holomorphic mappings},
  Ergebnisse der Mathematik und ihrer Grenzgebiete. 3. F, vol.~56, Springer,
  Heidelberg, 2011. \MR{2975791}

\bibitem[FR93]{FR}
Franc Forstneri\v{c} and Jean-Pierre Rosay, \emph{Approximation of
  biholomorphic mappings by automorphisms of {${\bf C}^n$}}, Invent. Math.
  \textbf{112} (1993), no.~2, 323--349. \MR{1213106 (94f:32032)}

\bibitem[IK12]{Ivarsson}
Bj{\"o}rn Ivarsson and Frank Kutzschebauch, \emph{Holomorphic factorization of
  mappings into {${\rm SL}_n(\mathbf C)$}}, Ann. of Math. (2) \textbf{175}
  (2012), no.~1, 45--69. \MR{2874639}

\bibitem[KK08]{KK-Zeit}
Shulim Kaliman and Frank Kutzschebauch, \emph{Density property for surfaces
  {$UV=P(\overline X)$}}, Math. Z. \textbf{258} (2008), no.~1, 115--131.
  \MR{2350038 (2008k:32062)}

\bibitem[KRP17]{K-R}
Frank Kutzschebauch and Alexandre Ramos-Peon, \emph{An oka principle for a
  parametric infinite transitivity property}, The Journal of Geometric Analysis
  \textbf{27} (2017), no.~3, 2018--2043.

\bibitem[RR88]{RosayRudin}
Jean-Pierre Rosay and Walter Rudin, \emph{Holomorphic maps from {${\bf C}^n$}
  to {${\bf C}^n$}}, Trans. Amer. Math. Soc. \textbf{310} (1988), no.~1,
  47--86. \MR{929658}

\bibitem[Ugo17]{Ugolini}
Riccardo Ugolini, \emph{A parametric jet-interpolation theorem for holomorphic
  automorphisms of $\mathbb{C}^n$}, The Journal of Geometric Analysis (2017),
  1--16.

\bibitem[Var00]{VarolinII}
Dror Varolin, \emph{The density property for complex manifolds and geometric
  structures. {II}}, Internat. J. Math. \textbf{11} (2000), no.~6, 837--847.
  \MR{1785520 (2002g:32027)}

\bibitem[Var01]{Varolin1}
Dror Varolin, \emph{The density property for complex manifolds and geometric
  structures}, J. Geom. Anal. \textbf{11} (2001), no.~1, 135--160. \MR{1829353
  (2002g:32026)}

\bibitem[Vas88]{Vaserstein}
L.~N. Vaserstein, \emph{Reduction of a matrix depending on parameters to a
  diagonal form by addition operations}, Proc. Amer. Math. Soc. \textbf{103}
  (1988), no.~3, 741--746. \MR{947649}

\bibitem[Win08]{WinkelmannTame}
J\"org Winkelmann, \emph{On tameness and growth conditions}, Doc. Math.
  \textbf{13} (2008), 97--101. \MR{2420907}

\bibitem[{Win}17]{WTame}
J.~{Winkelmann}, \emph{{Tame discrete subsets in Stein manifolds}}, ArXiv
  e-prints (2017).

\end{thebibliography}
\bibliographystyle{amsalpha}

\end{document}